\documentclass[11pt]{amsart}

\usepackage[margin=1.25in]{geometry}

\usepackage{graphicx}

\usepackage{wasysym}
\usepackage[misc]{ifsym}

\usepackage{amsmath, amssymb, amsthm, graphicx, enumerate, tikz, float, bbm, nicefrac, mathrsfs}
\usepackage{mathtools, comment}
\usepackage{subcaption}
\usepackage[colorlinks]{hyperref}
\hypersetup{citecolor=blue}
\usetikzlibrary{matrix,arrows,decorations.pathmorphing}

\newtheorem{theorem}{Theorem}[section]
\newtheorem{lemma}[theorem]{Lemma}
\newtheorem{corollary}[theorem]{Corollary}

\newcommand{\n}[1]{||#1||}

\theoremstyle{definition}
\newtheorem{definition}[theorem]{Definition}
\newtheorem{example}[theorem]{Example}

\newtheorem{conjecture}[theorem]{Conjecture}
\newtheorem*{solution*}{Solution}

\theoremstyle{remark}
\newtheorem{remark}[theorem]{Remark}
\newtheorem{claim}[theorem]{Claim}
\newtheorem*{pf*}{Pf}
\newtheorem*{pfbase*}{Proof of Base Case}
\newtheorem*{pfstep*}{Proof of Inductive Step}

\numberwithin{equation}{section}

\newcommand{\C}{\mathbb{C}}

\newcommand{\Z}{\mathbb{Z}}
\newcommand{\U}{\mathscr{U}}
\newcommand{\M}{\mathscr{M}}

\begin{document}

\allowdisplaybreaks

\title[Homeomorphisms on Homogeneous Spaces]{A Class of Homeomorphisms on Homogeneous Spaces of a Group Action}

\author{Samuel A. Hokamp}
\address{Department of Mathematics, Sterling College, Sterling, Kansas, 67579}
\email{samuel.hokamp@sterling.edu}


\date{\today}

\keywords{Group action, homogeneous space, locally convex topological vector space.\\\indent\emph{Corresponding author.} Samuel A. Hokamp \Letter~\href{mailto:samuel.hokamp@sterling.edu}{samuel.hokamp@sterling.edu}. \phone~920-634-7356.}

\begin{abstract}
We develop a class of homeomorphisms on a compact homogeneous space of a transitive group action and show how the class sheds new light on a decomposition problem. We further use this class to show that every such homogeneous space in a locally convex topological vector space which is also convex must necessarily be trivial, ie. a singleton set. Additionally, this class of homeomorphisms allows us to relate the induced group action on the space of continuous functions to the action on the homogeneous space. 
\end{abstract}

\maketitle

\section{Introduction}

In \cite{NR}, Nagel and Rudin establish a decomposition theorem for the closed unitarily invariant spaces of continuous and measurable functions on the unit sphere of $\mathbb{C}^n$. This work can also be found in \cite{RFT}. Using the work of Nagel and Rudin as an inspiration, we conjecture the following. Note that $C(X)$ denotes the continuous complex functions with domain $X$.

\begin{conjecture}\label{conjecture}
Let $G$ be a compact group acting continuously and transitively on a compact Hausdorff space $X$. Then there exists a 
collection $\mathscr{G}$ of closed, pairwise orthogonal, $G$-minimal subspaces of $C(X)$ induced by the group action such that any closed $G$-invariant subspace of $C(X)$ is the closure of the direct sum of some subcollection of $\mathscr{G}$.
\end{conjecture}

\noindent Definitions for \textit{$G$-minimal} and \textit{$G$-invariant} can be found in Section \ref{prelim}.

The existence of a collection $\mathscr{G}$ with the desired qualities is a result of the Peter-Weyl Theorem from \cite{peterweyl}. This is shown in Section \ref{Decomp} (Theorem \ref{collection}).
Recently in \cite{hokampconti}, the author proved Conjecture \ref{conjecture} subject to a condition on the collection $\mathscr{G}$. An attempt to drop this condition led to the study of the titular class of homeomorphisms on $X$ and the variety of subsequent results established in this paper.

Section \ref{Homies} is devoted to defining the class of homeomorphisms, establishing their properties, and proving some immediate results, the most important of which is the following:

\begin{theorem}\label{stabilizer thm}
If $\alpha\in G$ stabilizes any element of $X$, then $\alpha$ stabilizes every continuous function on $X$.
\end{theorem}

An interesting consequence of this result is the following relationship between the actions of $G$ on $X$ and $G$ on $C(X)$:

\begin{theorem}\label{actions}
If the action of $G$ on $C(X)$ is faithful, then the action of $G$ on $X$ is free.
\end{theorem}

In Section \ref{Decomp}, we revisit Conjecture \ref{conjecture} as discussed \cite{hokampconti}, wherein Corollary \ref{main corollary} of Section~\ref{Homies} reveals that the condition imposed in \cite{hokampconti} is more restrictive than originally thought, ultimately requiring that each space in $\mathscr{G}$ have dimension one, in which case Conjecture \ref{conjecture} is easily proven (Theorem 4.1 of \cite{hokampconti}).

Finally, the class of homeomorphisms also changes the perception of the general compact homogeneous space $X$. As a compact set, the natural inclination is to envision $X$ as a closed ball in $\mathbb{C}^n$, in part because the analogous work of Nagel and Rudin was on the sphere in $\mathbb{C}^n$, but this turns out to be the wrong image due to the ball's convexity:

\begin{theorem}\label{main result}
Let $X$ be a compact subset of a locally convex topological vector space, and $G$ a compact group acting continuously and transitively on $X$. If $X$ is convex, then $X$ is a singleton set.
\end{theorem}

\noindent That is, in a locally convex topological vector space, there exist no nontrivial compact homogeneous spaces that are convex. The proof of Theorem \ref{main result} is given in Section \ref{Convex Spaces}.

\section{Preliminaries}\label{prelim}


We begin by presenting some standard definitions concerning group actions. Let $X$ be a set and $G$ be a group that acts on $X$. That is, there exists a map $$\Phi:G\times X\to X$$ with the following properties:

\begin{itemize}
    \item[(1)] $\Phi(e,x)=x$ for all $x\in X$, where $e\in G$ is the identity element, and
    \item[(2)] $\Phi(\alpha,\Phi(\beta,x))=\Phi(\alpha\beta,x)$ for all $x\in X$ and $\alpha,\beta\in G$.
\end{itemize}

When we wish to be explicit, the map $\varphi_\alpha:X\to X$ shall denote the action of the element $\alpha\in G$ on $X$; that is, $\varphi_\alpha(x)=\Phi(\alpha,x)$ for $x\in X$. When there is no confusion, we shall write $\varphi_\alpha(x)$ simply as $\alpha x$.

Below are some basic properties of group actions found in a variety of classical texts.

\begin{definition}
Let $G$ be a group acting on a set $X$.
\begin{itemize}
    \item The action is \textit{faithful} if $\alpha x=x$ for all $x\in X$ implies that $\alpha=e$.
    \item The action is \textit{free} if $\alpha x=x$ for some $x\in X$ implies that $\alpha=e$.
    \item The action is \textit{transitive} if for each $x,y\in X$, there exists $\alpha\in G$ such that $y=\alpha x$.
\end{itemize}
\end{definition}

In proceeding, $G$ shall be a compact group with Haar measure $m$. We say that $G$ acts \textit{continuously} on a topological space $X$ if the map $\Phi$ is continuous with respect to the product topology. When $G$ acts continuously and transitively on a compact Hausdorff space $X$, each $\varphi_\alpha$ is a homeomorphism of $X$, and we say $X$ is a \textit{homogeneous space}. So $X$ shall be throughout the paper.

Let $C(X)$ denote the space of continuous complex functions with domain $X$ and uniform norm $\n{\cdot}$. The action of $G$ on $X$ induces an action of $G$ on $C(X)$ via the map $$(\alpha,f)\mapsto f\circ\varphi_\alpha^{-1}.$$ Note that this map need not be continuous, but it can be shown to be \textit{separately continuous}; that is, each map induced by fixing a component element is continuous: for $\alpha\in G$, the map $f\mapsto f\circ\varphi_\alpha^{-1}$ is a bijective linear isometry on $C(X)$, and for $f\in C(X)$, the continuity of the map $\alpha\mapsto f\circ\varphi_\alpha^{-1}$ is a consequence of Lemma~5.1 of \cite{hokampconti}. 

Throughout the paper, $\mu$ shall denote the unique regular Borel probability measure on $X$ that is invariant under the action of $G$. Specifically,
\begin{equation}\label{measinv}
    \int_X f\,d\mu=\int_X f\circ\varphi_\alpha\,d\mu,
\end{equation}
for all $f\in C(X)$ and $\alpha\in G$. The existence of such a measure is a result of Andr\'e Weil from \cite{weil1940}. A construction of $\mu$ can be found in \cite{joys} (Theorem~6.2), but existence can also be established non-constructively using the Riesz Representation Theorem (for reference, Theorem 6.19 \cite{RRC}). Since this technique will be useful to us in Section \ref{Homies}, we provide the result and its proof as Theorem \ref{G meas} in the appendix (Section \ref{appendix}).

For $1\leq p\leq\infty$, the notation $L^p(\mu)$ shall denote the usual Lebesgue spaces of measurable functions on $X$. 



To discuss the proof from \cite{hokampconti} in Section \ref{Decomp}, we need the following definitions:

\begin{definition}
A space of complex functions $Y$ defined on $X$ is \textbf{invariant under $G$} (\textbf{$G$-invariant}) if $f\circ \varphi_\alpha\in Y$ for every $f\in Y$ and every $\alpha\in G$.
\end{definition}

\begin{remark}
Since the action of $G$ on $X$ is continuous, $C(X)$ is $G$-invariant. Conversely, if $C(X)$ is $G$-invariant, then each action $\varphi_\alpha$ must be continuous.
\end{remark}

\begin{remark}
Explicitly, the invariance property \eqref{measinv} means $\mu(\alpha E)=\mu(E)$ for every Borel set $E$ and every $\alpha\in G$. Consequently, property \eqref{measinv} holds for every $L^p$ function, and thus $L^p(\mu)$ is $G$-invariant for all $1\leq p\leq\infty$.
\end{remark}


\begin{definition}
A space $Y\subset C(X)$ is \textbf{$G$-minimal} if it is $G$-invariant and contains no nontrivial $G$-invariant spaces.
\end{definition}

Finally, all classical results used in the paper can be found in a variety of texts, with the reference given being only one option.

\section{A Class of Homeomorphisms on Homogeneous Spaces}\label{Homies}

In this section we establish the titular class of homeomorphisms on $X$. Each homeomorphism $\phi_x$ of this class is induced by an element of $x\in X$ and is constructed as follows. Fix $x\in X$. We define the set $$S(x)=\{\alpha\in G:\alpha x=x\}.$$ This is the set of elements in $G$ that stabilize $x$, which is a closed subgroup of $G$ called the \textit{isotopy subgroup} for $x$. A theorem of Weil states that $X$ is homeomorphic to $G/S(x)$ with the quotient topology (for reference, Theorem 6.1 of \cite{joys}), the homeomorphism $\phi$ given by $$\phi(\alpha x)=\alpha S(x).$$ Note the importance of transitivity here. Further, the map $\iota:G/S(x)\to G/S(x)$ given by $$\iota(\alpha S(x))=\alpha^{-1}S(x)$$ is a homeomorphism on $G/S(x)$ in the quotient topology\footnote{In fact, using the strong continuity of the quotient map, one can show that if $G$ is a topological group and $H$ a subgroup, then for any homeomorphism $f$ of $G$, the map $$xH\mapsto f(x)H$$ is a homeomorphism of $G/H$.}. We compose the homeomorphisms $\phi$, $\iota$, and $\phi^{-1}$ together to get the homeomorphism $\phi_x$: $$\phi_x=\phi^{-1}\circ\iota\circ\phi.$$ In particular, for all $\alpha\in G$, $$\phi_x(\alpha x)=\alpha^{-1}x,$$ so that $\phi_x(y)=\beta^{-1}x$, where $y\in X$ is such that $y=\beta x$.

Theorem \ref{properties} now establishes some basic properties of the maps $\phi_x$:

\begin{theorem}\label{properties}
For each $x\in X$, let $\phi_x$ denote the corresponding homeomorphism defined previously. Then,
\begin{itemize}
    \item[(1)] $\phi_x(x)=x$.
    \item[(2)] $\phi_x^{-1}=\phi_x$.
    \item[(3)] For all $\alpha\in G$, $\phi_{\alpha x}=\varphi_\alpha\circ\phi_x\circ\varphi_{\alpha^{-1}}$
    \item[(4)] For all $\alpha\in G$, $\phi_{\alpha x}=\phi_x\circ\phi_{\alpha^{-1}x}\circ\phi_x$.
\end{itemize}
\end{theorem}

\begin{proof}
To prove (1), let $\alpha\in G$ be any element such that $\alpha x=x$. The identity element will work fine. Observe then $$\phi_x(x)=\phi_x(\alpha x)=\alpha^{-1}x=x,$$ since $\alpha^{-1}x=x$ if and only if $\alpha x=x$. To show (2), let $y\in X$ and $\beta\in G$ such that $y=\beta x$. Then $$\phi_x(\phi_x(y))=\phi_x(\phi_x(\beta x))=\phi_x(\beta^{-1}x)=\beta x=y.$$ To show (3), fix $\alpha\in G$, and let $y\in X$ and $\beta\in G$ such that $y=\beta x$. Observe then $$\alpha\phi_x(\alpha^{-1}y)=\alpha\phi_x(\alpha^{-1}\beta x)=\alpha\beta^{-1}\alpha x=\phi_{\alpha x}(\beta\alpha^{-1}\alpha x)=\phi_{\alpha x}(\beta x)=\phi_{\alpha x}(y).$$ To show (4), let $\alpha$, $y$, and $\beta$ be as defined above. Then $$\phi_{x}(\phi_{\alpha^{-1}x}(\phi_x(y)))=\phi_{x}(\phi_{\alpha^{-1}x}(\phi_x(\beta x)))=\phi_{x}(\phi_{\alpha^{-1}x}(\beta^{-1}x))$$ $$=\phi_{x}(\phi_{\alpha^{-1}x}(\beta^{-1}\alpha\alpha^{-1}x))=\phi_x(\alpha^{-1}\beta\alpha^{-1}x)=\alpha\beta^{-1}\alpha x=\phi_{\alpha x}(y),$$ by jumping into the middle of the equality proving (3).
\end{proof}

A very useful result of the homeomorphisms $\phi_x$ is that each one induces a bijective linear isometry $U_x$ on each space $C(X)$ or $L^p(\mu)$, $1\leq p\leq\infty$, via $U_xf=f\circ\phi_x$.

\begin{theorem}
Let $\mathscr{X}$ be any of the spaces $C(X)$ or $L^p(\mu)$, for $1\leq p\leq\infty$, and $x\in X$. The map $U_x:\mathscr{X}\to\mathscr{X}$ given by $U_xf=f\circ\phi_x$ is a bijective linear iosmetry.
\end{theorem}

\begin{proof}
The bijectivity of each $U_x$ is clear, since each is its own inverse by Theorem \ref{properties}(2), and linearity is clear as well. To show that $U_x$ is an isometry on $C(X)$, observe that $$|(U_xf)(y)|=|f(\phi_x(y))|\leq\n{f}\hspace{.25in}\text{and}\hspace{.25in}|f(y)|=|(U_xf)(\phi_x(y))|\leq\n{U_xf}$$ for all $y\in X$ and $f\in C(X)$. In particular, $\n{U_x f}=\n{f}$.

To show that $U_x$ is an isometry on each $L^p(\mu)$, for $1\leq p\leq\infty$, we show that
\begin{equation}\label{Ux inv}
\int_XU_xf\,d\mu=\int_Xf\,d\mu
\end{equation}
holds for all $f\in C(X)$. From the proof of Theorem \ref{G meas} in the appendix, we have $$\int_XU_xf\,d\mu=\int_G(U_xf)(\gamma x)\,dm(\gamma)=\int_Gf(\phi_x(\gamma x))\,dm(\gamma)=\int_Gf(\gamma^{-1}x)\,dm(\gamma),$$ for $f\in C(X)$, where $m$ denotes the Haar measure on $G$. However, the Haar measure on $G$ is invariant under inversion in $G$, and thus $$\int_Gf(\gamma^{-1}x)\,dm(\gamma)=\int_Gf(\gamma x)\,dm(\gamma)=\int_Xf\,d\mu.$$ This verifies property \eqref{Ux inv}. We then apply property \eqref{Ux inv} alongside Lusin's Theorem and the Dominated Convergence Theorem (for reference, Theorem 2.24 and Theorem 1.34, respectively, of \cite{RRC}) to get that $\mu(\phi_x(E))=\mu(E)$ for all Borel sets $E\subset X$. Consequently, $$\int_XU_xf\,d\mu=\int_X f\,d\mu$$ holds for every $f\in L^p(\mu)$, for $1\leq p\leq\infty$. That $U_x$ is an isometry follows directly.
\end{proof}

The following theorem establishes some basic properties of the operators $U_x$, the third of which is the key piece in the results that follow.

\begin{theorem}\label{operator properties}
Suppose $\mathscr{X}$ is any of the spaces $C(X)$ or $L^p(\mu)$, for $1\leq p\leq\infty$, $x\in X$, and $\alpha\in G$. For all $f\in\mathscr{X}$, we have that
\begin{itemize}
    \item[(1)] $\displaystyle (U_{\alpha x}f)\circ\varphi_\alpha=U_{x}(f\circ\varphi_\alpha)$,
    \item[(2)] $\displaystyle U_{\alpha x}f=U_xU_{\alpha^{-1}x}U_xf$, and
    \item[(3)] $(U_xf)\circ\varphi_\alpha=U_xf$ when $\alpha x=x$.
\end{itemize}
\end{theorem}

\begin{proof}
Equalities (1) and (2) follow from (3) and (4) of Theorem \ref{properties}, respectively: $$(U_{\alpha x}f)\circ\varphi_\alpha=f\circ\phi_{\alpha x}\circ\varphi_\alpha=f\circ\varphi_\alpha\circ\phi_{x}=U_{x}(f\circ\varphi_\alpha).$$ $$U_{\alpha x}f=f\circ\phi_{\alpha x}=f\circ\phi_x\circ\phi_{\alpha^{-1}x}\circ\phi_x=U_xU_{\alpha^{-1}x}U_xf.$$

To show (3), we note that since $\alpha x=x$, we have $\alpha^{-1}x=x$. Let $y\in X$ and $\beta\in G$ such that $y=\beta x$. We then have that $$[(U_xf)\circ\varphi_\alpha](y)=f(\phi_x(\alpha\beta x))=f(\beta^{-1}\alpha^{-1}x)=f(\beta^{-1}x)=f(\phi_x(\beta x))=(U_xf)(y).$$ Since this can be done for any $y\in X$, the proof is complete.
\end{proof}

Replacing $f$ with $U_xf$ in Theorem \ref{operator properties}(3) yields Theorem \ref{stabilizer thm}. A more particular statement of this result is given below as Corollary \ref{main corollary}.

\begin{corollary}\label{main corollary}
If $\alpha\in G$ and $x\in X$ such that $\alpha x=x$, then $f=f\circ\varphi_{\alpha}$ for all $f\in C(X)$.
\end{corollary}

Theorem \ref{actions}, relating the actions of $G$ on $X$ and $G$ on $C(X)$, immediately follows.

\begin{proof}[Proof of Theorem \ref{actions}]
Suppose that the action of $G$ on $C(X)$ is faithful, and let $\alpha\in G$ and $x\in X$ such that $\alpha x=x$. To show that the action of $G$ on $X$ is free, we must prove that $\alpha$ is the identity element of $G$.

Since $\alpha$ stabilizes some element of $X$, we must have that $f=f\circ\varphi_\alpha$ for all $f\in C(X)$ by Corollary \ref{main corollary}. Since the action of $G$ on $C(X)$ is faithful, we get that $\alpha$ must be the identity element of $G$, as desired.
\end{proof}

\section{Discussing a Decomposition Problem}\label{Decomp}

In this section, we revisit Conjecture \ref{conjecture} as studied in \cite{hokampconti} and discuss the additional condition imposed on $X$ that allowed for the main result of that paper. To begin, we show how the Peter-Weyl Theorem from \cite{peterweyl} is used to establish the existence of a 
collection $\mathscr{G}$ of closed $G$-minimal spaces of continuous functions over $X$:

\begin{theorem}\label{collection}
Let $G$ be a compact group acting continuously and transitively on a compact Hausdorff space $X$. Then there exists a 
collection $\mathscr{G}$ of closed, pairwise orthogonal, $G$-minimal subspaces of $C(X)$ whose direct sum is $L^2(\mu)$.
\end{theorem}

The statement of the key part of the Peter-Weyl Theorem used in the proof of Theorem~\ref{collection} (given below as Theorem \ref{part2}) is taken from \cite{knapp}. One can also find this result as Theorem~5.2 from \cite{folland}.

\begin{theorem}[\cite{knapp} Theorem 1.12(e)]\label{part2}
Let $\Phi$ be a unitary representation of $G$ on a Hilbert space $V$. Then $V$ is the orthogonal direct sum of finite-dimensional irreducible invariant subspaces.
\end{theorem}

Further treatment of unitary representations can be found in \cite{folland}. Note that a subspace $M\subset V$ is \textit{invariant} for $\Phi$ if $\Phi(\alpha)M\subset M$ for all $\alpha\in G$. Further, a subspace of $V$ is \textit{irreducible} if it contains no non-trivial invariant subspaces.

Lemma 5.3 of \cite{hokampconti}, given below as Lemma \ref{lemma}, is also used in the proof of Theorem \ref{collection}.

\begin{lemma}[\cite{hokampconti} Lemma 5.3]\label{lemma}
Let $\mathscr{X}$ be any of the spaces $C(X)$ or $L^p(\mu)$, for $1\leq p<\infty$. If $Y$ is a closed $G$-invariant space in $\mathscr{X}$, then $Y\cap C(X)$ is dense in $Y$.
\end{lemma}

\begin{proof}[Proof of Theorem \ref{collection}]
One easily checks that the map $\Phi$ given by $$\Phi:G\to \mathscr{U}(L^2(\mu))\hspace{.125in}\text{such that}\hspace{.125in}\Phi(\alpha)f=f\circ\varphi_\alpha\hspace{.125in}\text{for all }f\in L^2(\mu),$$ is a unitary representation of $G$, where $\mathscr{U}(L^2(\mu))$ denotes the group of unitary operators on $L^2(\mu)$. Thus, by the part of the Peter-Weyl Theorem given above as Theorem~\ref{part2}, $L^2(\mu)$ decomposes into the orthogonal direct sum of finite-dimensional irreducible invariant subspaces. To translate into our terminology, an invariant subspace of the representation is equivalent to the subspace being $G$-invariant, and a subspace being irreducible is equivalent to it being $G$-minimal.

Let $\mathscr{G}$ denote this collection of pairwise orthogonal, $G$-minimal spaces. Since each element of $\mathscr{G}$ is finite-dimensional, each is closed. Finally, Lemma \ref{lemma} yields that each element of $\mathscr{G}$ is in fact a space of continuous functions. Note the role that the finite-dimensionality of the elements of $\mathscr{G}$ plays here. 
\end{proof}


In \cite{hokampconti}, the collection $\mathscr{G}$ was affixed with the additional property that
\begin{equation}\label{prop from hokampconti}
\dim (H\cap H(x))=1
\end{equation}
for each $x\in X$ and each $H\in\mathscr{G}$, where $$H(x)=\{f\in C(X):f=f\circ\varphi_\alpha,\text{ for all }\alpha\in G\text{ such that }\alpha x=x\}.$$ That is, $H(x)$ is the space of continuous functions on $X$ that are stabilized in $C(X)$ by every stabilizer of $x$ from $G$. However, Corollary \ref{main corollary} immediately yields the following:

\begin{theorem}
For each $x\in X$, we have that $H(x)=C(X)$.
\end{theorem}

Thus, property \eqref{prop from hokampconti} amounts to requiring each space $H\in\mathscr{G}$ to have dimension 1: $$1=\dim (H\cap H(x))=\dim (H\cap C(X))=\dim H.$$ This property occurs in some important cases, such as when $G$ is abelian (\cite{folland}, Corollary~3.6), but is not the case in general, as we see in the work of Nagel and Rudin in \cite{NR}. Further, the work of Nagel and Rudin shows that this condition is not strictly necessary to achieve the decomposition of closed $G$-invariant spaces into $G$-minimal spaces. Currently, Conjecture~\ref{conjecture} in its full generality remains unproven.

\section{Convex Homogeneous Spaces of a Group Action}\label{Convex Spaces}

In this section, we prove Theorem \ref{main result}, the proof of which uses Corollary \ref{main corollary} alongside the Schauder-Tychonoff Theorem (for reference, Theorem 5.28 of \cite{RFA}) and Urysohn's Lemma (for reference, Theorem 33.1 of \cite{MUNK}). Note that Urysohn's Lemma applies because $X$, as a compact Hausdorff space, is a normal space (Theorem 32.3 of \cite{MUNK}).

\begin{proof}[Proof of Theorem \ref{main result}]
To begin, let $G$ be the compact group which acts continuously and transitively on $X$ to make it a homogeneous space. We note that since $X$ is a convex, compact subset of a locally convex topological vector space, every homeomorphism of $X$ has a fixed point, by the Schauder-Tychonoff Theorem. In particular, for every $\alpha\in G$, the map $\varphi_\alpha$ has a fixed point; that is, for each $\alpha\in G$ there exists $x\in X$ such that $\alpha x=x$. Applying Corollary \ref{main corollary}, for all $f\in C(X)$ we have that $f=f\circ\varphi_\alpha$.

However, due to the transitivity of the group action, we get that each function $f\in C(X)$ must be constant: Let $x,y\in X$ and $\alpha\in G$ such that $y=\alpha x$. Then $$f(y)=f(\alpha x)=(f\circ\varphi_\alpha)(x)=f(x).$$ This can be done for each pair of points in $X$, so $f$ must be constant.

Finally, if $X$ were to contain more than one point, for each pair of distinct points in $X$ there would exist a continuous function separating them, according to Urysohn's Lemma. However, as we just established, all continuous functions on $X$ are constant, meaning $X$ cannot contain any pairs of distinct points. Thus, $X$ must be a singleton set.
\end{proof}

\section{Appendix: Existence of a \texorpdfstring{$G$}{G}-Invariant Measure on \texorpdfstring{$X$}{X}}\label{appendix}

In this section we present a non-constructive proof of the following result of Andr\'e Weil from \cite{weil1940}. Recall $m$ denotes the Haar measure of $G$.

\begin{theorem}\label{G meas}
There exists a unique regular Borel probability measure $\mu$ on $X$ that is invariant under $G$ in the sense that
\begin{equation}\label{inv}
    \int_X f\,d\mu=\int_X f\circ\varphi_\alpha\,d\mu,
\end{equation}
for every $f\in C(X)$ and every $\alpha\in G$.
\end{theorem}

\begin{proof}

We define the map $\Lambda$ on $C(X)$ by $$\Lambda f=\int_G f(\gamma x)\,dm(\gamma),$$ for some $x\in X$. The transitivity of the action and the invariance of the Haar measure yield that $\Lambda$ is well-defined: If $y\in X$, then there exists $\beta\in G$ such that $\beta x=y$. Thus, $$\int_G f(\gamma y)\,dm(\gamma)=\int_G f(\gamma \beta x)\,dm(\gamma)=\int_G f(\gamma x)\,dm(\gamma).$$

Clearly, $\Lambda$ is a linear functional on $C(X)$. Further, it is bounded: $$|\Lambda f|\leq\int_G |f(\gamma x)|\,dm(\gamma)\leq\n{f}.$$

Thus, by the Riesz Representation Theorem (Theorem 6.19 \cite{RRC}), there exists a unique regular Borel measure $\mu$ on $X$ such that $$\int_X f\,d\mu=\Lambda f$$ for all $f\in C(X)$. Since $m$ is a probability measure on $G$, $\mu$ is a probability measure on $X$.

The invariance property \eqref{inv} follows from the invariance of the Haar measure: $$\Lambda(f\circ\varphi_\alpha)=\int_G f(\alpha\gamma x)\,dm(\gamma)=\int_G f(\gamma x)\,dm(\gamma)=\Lambda f.$$

\end{proof}

\section{Data Availability Statement}

Data sharing not applicable to this article as no datasets were generated or analysed during the current study.

\bibliographystyle{unsrt}
\bibliography{bibliography}

\end{document}